\documentclass[12pt]{amsart}

\usepackage{hyperref}
\usepackage{graphicx}
\usepackage{enumerate}
\usepackage{vmargin}
\setpapersize{USletter}
\setmargrb{1in}{1in}{1in}{1in} % --- top, left, right, bottom margins
\newtheorem*{thm*}{Theorem}%[section]
\newtheorem{thm}{Theorem}[section]

\newtheorem*{subadditivitythm}{Subadditivity Theorem}

\theoremstyle{definition}
\newtheorem{example}[thm]{Example}
\newtheorem{remark}[thm]{Remark}

\newtheorem{prop}[thm]{Proposition}

\newtheorem{lem}[thm]{Lemma}

\newtheorem{property}[thm]{Property}

\DeclareMathOperator{\lct}{lct}
\DeclareMathOperator{\mult}{mult}
\DeclareMathOperator{\Zeros}{Zeros}
\DeclareMathOperator{\bight}{bight}
\DeclareMathOperator{\height}{ht}
\DeclareMathOperator{\conv}{conv}
\DeclareMathOperator{\ord}{ord}

\newcommand{\C}{\mathbb{C}}

\newcommand{\pr}[1]{\mathbb{P}^{#1}}
\newcommand{\symbolicpower}[2]{{#1}^{(#2)}}
\newcommand{\symbi}{\symbolicpower{I}{\bullet}}
\newcommand{\symbsystem}[1]{\symbolicpower{#1}{\bullet}}
\newcommand{\multideal}{\mathcal{J}}
\newcommand{\asymptoticmultideal}[2]{\multideal \left( #1 \cdot #2 \right)}
\newcommand{\asympt}[2]{\asymptoticmultideal{#1}{\symbsystem{#2}}}
\newcommand{\defining}[1]{\textbf{#1}}

\DeclareMathOperator{\Int}{Int}

\DeclareMathOperator{\codim}{codim}
\renewcommand{\a}{\mathfrak{a}}

\DeclareMathOperator{\monom}{monom}
\DeclareMathOperator{\Newt}{Newt}

\title{Bounding symbolic powers via asymptotic multiplier ideals}
\author{Zach Teitler}
\email{zteitler@math.tamu.edu}
\address{Department of Mathematics \\ Mailstop 3368 \\ Texas A\&M University \\ College Station, TX 77843-3368 \\ USA}
\date{September 22, 2009}
\thanks{Acknowledgements: I am grateful to Brian Harbourne for inviting me to write this material
(originally as an appendix to lecture notes~\cite{harbourne-lecture-notes}
for a course he gave at a summer school in Cracow in March, 2009)
and for numerous helpful conversations.}

\begin{document}

\bibliographystyle{amsalpha}       % Set the bibliography style to AMS
                                % alphabetized. (Can use ``amsalpha'' or
                                % ``abbrv''instead.) [amsplain, plain]

\maketitle

We revisit a bound on symbolic powers found by Ein--Lazarsfeld--Smith and subsequently improved by Takagi--Yoshida.
We show that the original argument of~\cite{els:symbolic-powers} actually gives the same improvement.
On the other hand, we show by examples that any further improvement based on the same technique appears unlikely.
This is primarily an exposition; only some examples and remarks might be new.

\section{Uniform bounds for symbolic powers}

For a radical ideal $I$, the symbolic power $\symbolicpower{I}{p}$ is the collection of elements
that vanish to order at least $p$ at each point of $\Zeros(I)$.
If $I$ is actually prime, then $\symbolicpower{I}{p}$ is the $I$-associated primary component of $I^p$;
if $I$ is only radical, writing $I = C_1 \cap \dots \cap C_s$ as an intersection of prime ideals,
$\symbolicpower{I}{p} = \symbolicpower{C}{p}_1 \cap \dots \cap \symbolicpower{C}{p}_s$. %
The inclusion $I^p \subseteq \symbolicpower{I}{p}$ always holds,
but the reverse inclusion holds only in some special cases,
such as when $I$ is a complete intersection.

Swanson~\cite{MR1778408} showed that for rings $R$ satisfying a certain hypothesis,
for each ideal $I$, there is an integer $e=e(I)$ such that the symbolic power
$\symbolicpower{I}{e r} \subseteq I^r$ for all $r\geq0$.
Ein--Lazarsfeld--Smith~\cite{els:symbolic-powers} showed that in a regular local ring $R$
in equal characteristic $0$ and for $I$ a radical ideal,
one can take $e(I)=\bight(I)$, the \defining{big height} of $I$,
which is the maximum of the codimensions
of the irreducible components of the closed subset of zeros of $I$.
In particular, $\bight(I)$ is at most the dimension of the ambient space,
so $e=\dim R$ is a single value that works for all ideals.
More generally, for any $k\geq0$, $\symbolicpower{I}{er+kr} \subseteq (\symbolicpower{I}{k+1})^r$ for all $r\geq1$.
Very shortly thereafter, Hochster--Huneke~\cite{MR1881923} generalized this result by characteristic $p$ methods.

It is natural to regard these results in the form $\symbolicpower{I}{m} \subseteq I^r$
for $m \geq f(r) = er$, $e=\bight(I)$.
Replacing $f(r)=er$ with a smaller function would give a stronger bound on symbolic powers
(containment in $I^r$ would begin sooner).
So it is natural to ask, how far can one reduce the bounding function $f(r)=er$?

Bocci--Harbourne~\cite{bocci-harbourne} introduced the \defining{resurgence} of $I$,
$\rho(I) = \sup\{ m/r : \symbolicpower{I}{m} \not\subseteq I^r \}$.
Thus if $m > \rho(I) r$, $\symbolicpower{I}{m} \subseteq I^r$.
The Ein--Lazarsfeld--Smith and Hochster--Huneke results show $\rho(I) \leq \bight(I) \leq \dim R$.
It can be smaller.
For example, if $I$ is smooth or a reduced complete intersection, $\rho(I)=1$.
More interestingly, Bocci--Harbourne~\cite{bocci-harbourne}
show that if $I$ is an ideal of $n$ reduced points in general position
in $\pr{2}$, $\rho(I) = \rho_n \leq 3/2$.
On the other hand, Bocci--Harbourne show for each $n$, $1 \leq e \leq n$,
and $\epsilon > 0$ there are ideals $I$ on $\pr{n}$ with $\bight(I)=e$
such that $\rho(I) > e-\epsilon$.
This suggests that one cannot expect improvement
in the slope of the linear bound $m \geq er$, at least not in very general terms.
So one naturally turns toward the possibility of subtracting a constant term.

Huneke raised the question of whether, for $I$ an ideal of reduced points in $\pr{2}$,
$\symbolicpower{I}{3} \subseteq I^2$.
Bocci--Harbourne's result $\rho \leq 3/2$ gives an affirmative answer to Huneke's question, and much more,
for points in general position.
Some other cases have been treated, e.g., points on a conic, but the general case---i.e.,
points in arbitrary position---remains open.

A conjecture of Harbourne (Conjecture~8.4.3 in \cite{seshadri}) states that for a homogeneous ideal $I$ on $\pr{n}$,
$\symbolicpower{I}{m} \subseteq I^r$ for all $m \geq nr-(n-1)$,
and even stronger, that the containment holds
for all $m \geq er-(e-1)$, where $e=\bight(I)$.
Huneke's question would follow at once as the case $n=e=r=2$.

Some results in this direction have been obtained by various authors.
Huneke has observed that Harbourne's conjecture holds in characteristic $p$
for values $r = p^k$, $k \geq 1$, see Example~IV.5.3 of \cite{harbourne-lecture-notes}
or Example~8.4.4 of \cite{seshadri}.
Takagi--Yoshida~\cite{takagi-yoshida} and Hochster--Huneke independently
showed by characteristic $p$ methods
that $\symbolicpower{I}{er+kr-1} \subseteq (\symbolicpower{I}{k+1})^r$ for all $k\geq 0$ and $r \geq 1$
when $I$ is $F$-pure (see below).
More generally, Takagi--Yoshida show a characteristic $p$ version of the following:
\begin{thm}[\cite{takagi-yoshida}]\label{thm:main}
Let $R$ be a regular local ring of equal characteristic $0$, $I \subseteq R$ a reduced ideal,
$e = \bight(I)$ the greatest height of an associated prime of $I$,
and $\ell$ an integer, $0 \leq \ell < \lct(\symbi)$ where $\lct(\symbi)$
is the log canonical threshold of the graded system of symbolic powers of $I$, see below.
Then $\symbolicpower{I}{m} \subseteq I^r$ whenever $m \geq er - \ell$.
More generally, for any $k \geq 0$, $\symbolicpower{I}{m} \subseteq (\symbolicpower{I}{k+1})^r$
whenever $m \geq er+kr-\ell$.
\end{thm}
This statement is a slight modification of Remark 3.4 of \cite{takagi-yoshida}.

The Ein--Lazarsfeld--Smith uniform bounds on symbolic powers described above are the case $\ell=0$.
The $F$-pure case implies $\lct(\symbi) > 1$, so we may take $\ell=1$.
(More precisely, $F$-pure means $\lct(I)>1$, and we will see $\lct(\symbi) \geq \lct(I)$.)

The idea of the proof is to produce an ideal $J$ with
$\symbolicpower{I}{m} \subseteq J$ and $J \subseteq (\symbolicpower{I}{k+1})^r$.
Ein--Lazarsfeld--Smith introduced asymptotic multiplier ideals in~\cite{els:symbolic-powers}
and, among other results, proved the uniform bounds described above by taking $J$ to be an asymptotic multiplier ideal.
For Takagi--Yoshida the ideal $J$ is a generalized test ideal, a characteristic $p$ analogue
of the asymptotic multiplier ideals introduced by Hara--Yoshida~\cite{MR1974679}.

In this note, $J$ will be an asymptotic multiplier ideal.
We will review multiplier ideals in \S\ref{sect: mult-ideals}
and discuss some examples in \S\ref{sect: examples}:
the asymptotic multiplier ideals of monomial ideals and hyperplane arrangements.
We will revisit the argument given by Ein--Lazarsfeld--Smith in the case $\ell=0$
to show that it actually gives Theorem~\ref{thm:main} in \S\ref{sect: proof}.

In \S\ref{sect: non-improvement} we consider two ways in which
the argument of \S\ref{sect: proof} falls short of the improved bounds we hope for.
First, the condition $0 \leq \ell < \lct(\symbi)$, while generalizing the result of \cite{els:symbolic-powers},
is nevertheless quite restrictive.
Second, the argument of~\cite{els:symbolic-powers} actually produces two ideals,
$\symbolicpower{I}{m} \subseteq J_1 \subseteq J_2 \subseteq (\symbolicpower{I}{k+1})^r$.
We will consider as an example the ideal $I=(xy,xz,yz)$ of the union of the three coordinate axes in $\C^3$.
We will show that in this example the first and last inclusions are actually equalities,
while the middle inclusion $J_1 \subseteq J_2$ is very far.
So if any improvement remains to be found, one must consider the middle inclusion.

\section{Multiplier ideals}\label{sect: mult-ideals}

Henceforth we fix $X = \C^n$ and consider ideals in the ring $R = \C[x_1,\dots,x_n]$.

Note that for a prime homogeneous ideal $I$, a homogeneous form $F$ vanishes to order $p$
along the projective variety defined by $I$ in $\pr{n-1}$ if and ony if it vanishes to order $p$ on the affine variety
defined by $I$ in $\C^n$.
In this way the Bocci--Harbourne results and Huneke question for points in $\pr{2}$
translate to questions about symbolic powers of (homogeneous) ideals in the affine setting.

\subsection{Ordinary multiplier ideals}

To an ideal $I \subseteq \C[x_1,\dots,x_n]$, regarded as a sheaf of ideals on $X = \C^n$,
and a real parameter $t \geq 0$ one may associate
the \defining{multiplier ideal} $\multideal(I^t) \subseteq \C[x_1,\dots,x_n]$.
The multiplier ideals are defined in terms of a resolution of singularities of $I$.
For details, see~\cite{MR2132649,pag2}.

Note, in the notation $\multideal(I^t)$ the $t$ indicates dependance on the parameter $t$, rather than a power of $I$.
In particular, $\multideal(I^t)$ is defined for all real $t \geq 0$, whereas $I^t$ on its own only makes sense for integer $t \geq 0$.
See, however, Property~\ref{property: integer powers}.

Rather than present the somewhat involved definition here, we give a short list of properties of multiplier ideals
which are all that we will use.
(The reader may take these as axioms, although the properties listed here
do not characterize multiplier ideals.)

\begin{property}\label{property: dependence on arguments}
For any nonzero ideal $I$, $\multideal(I^0) = (1)$, the unit ideal.
As the parameter $t$ increases, the multiplier ideals get smaller: if $t_1 < t_2$ then $\multideal(I^{t_1}) \supseteq \multideal(I^{t_2})$.

On the other hand, if $I_1 \subseteq I_2$ then $\multideal(I_1^t) \subseteq \multideal(I_2^t)$.

%Multiplier ideals of the unit ideal are again the unit ideal: $\multideal((1)^t) = (1)$ for all $t \geq 0$.
%By convention, multiplier ideals of the zero ideal are zero: $\multideal((0)^t) = (0)$ for all $t \geq 0$.

Thus multiplier ideals, as functions of two arguments, are ``order-preserving'' in the ideal and ``order-reversing'' in the real parameter.
\end{property}

\begin{property}\label{property: integer powers}
For any real number $t \geq 0$ and integer $k > 0$, $\multideal(I^{kt}) = \multideal((I^k)^t)$.
\end{property}

\begin{property}
For any $t \geq 0$ and integer $p \geq 0$, $I^p \multideal(I^t) \subseteq \multideal(I^{p+t})$.
See Proposition~9.2.32(iv) of~\cite{pag2}.
\end{property}

\begin{property}\label{property: multiplier ideal of smooth}
When $\Zeros(I)$ is smooth and irreducible
with codimension $\codim(\Zeros(I)) = e = \bight(I)$,
$\multideal(I^t) = I^{\lfloor t \rfloor - e + 1}$.
In particular, $\multideal(I^t) \subseteq I$ for $t \geq e$.
More generally, if $I$ is reduced and $\Zeros(I) = V_1 \cup \dots \cup V_s$,
then restricting to a neighborhood of a general point on each $V_i$, we see $\multideal(I^t)$ vanishes on $V_i$ for $t \geq \codim V_i$,
hence $\multideal(I^t) \subseteq I$ for $t \geq \max \codim V_i = \bight(I)$.
\end{property}

The above list is a small selection of the many interesting properties of multiplier ideals.
See~\cite{MR2132649,pag2} for more, including excellent expositions of the definition
(from which all the above properties follow immediately).
Among these many other properties we single out one which we will use here,
due to Demailly--Ein--Lazarsfeld~\cite{MR1786484}.
\begin{subadditivitythm}
$\multideal(I^{t_1 + t_2}) \subseteq \multideal(I^{t_1})\multideal(I^{t_2})$.
In particular for any integer $r \geq 0$, $\multideal(I^{rt}) \subseteq \multideal(I^t)^r$.
\end{subadditivitythm}

\subsection{Asymptotic multiplier ideals}

A \defining{graded system} of ideals $\a_{\bullet} = \{ \a_n \}_{n=1}^{\infty}$
is a collection of ideals satisfying $\a_p \a_q \subseteq \a_{p+q}$,
and (to avoid trivialities) at most finitely many of the $\a_n$ may be zero.
Note that $\a_p$, $\a_{p+1}$ are not required to satisfy any particular relation,
but $(\a_p)^k \subseteq \a_{kp}$.
By convention, $\a_0 = \C[x_1,\dots,x_n]$, so that $\bigoplus_{n=0}^{\infty} \a_n$ is a $\C[x_1,\dots,x_n]$-algebra.
A trivial graded system is one of the form $\a_n = \a_1^n$.
Our main interest will be in the graded system of symbolic powers of a (reduced) ideal $I$,
$\symbi = \{ \symbolicpower{I}{n} \}_{n \geq 0}$.

To a graded system $\a_{\bullet}$ and real parameter $t \geq 0$ one can associate an \defining{asymptotic multiplier ideal}
$\multideal(\a_{\bullet}^t)$, or $\asympt{t}{I}$, defined by
\[
  \multideal(\a_{\bullet}^t) = \max_{p \geq 1} \multideal(\a_p^{t/p}) .
\]
This definition was given in~\cite{els:symbolic-powers}.
We must justify the existence and well-definedness of this maximum; we repeat the argument of~\cite{els:symbolic-powers}.
Note that since $(\a_p)^q \subseteq \a_{qp}$, by the properties of multiplier ideals we have
\[
  \multideal(\a_p^{t/p}) = \multideal((\a_p^q)^{t/pq}) \subseteq \multideal(\a_{pq}^{t/pq}) .
\]
The Noetherian property assures that among the ideals $\multideal(\a_p^{t/p})$, one is a (relative) maximum.
If $\multideal(\a_p^{t/p})$ is a maximum, then by the above, $\multideal(\a_p^{t/p}) = \multideal(\a_{pq}^{t/pq})$.
Hence if $\multideal(\a_p^{t/p})$ and $\multideal(\a_q^{t/q})$ both are maxima, then they coincide with each other.
Thus there is a unique maximum of this collection of ideals.

In particular, $\multideal(\a_{\bullet}^t) = \multideal(\a_p^{t/p})$
for $p \gg 0$ and sufficiently divisible; i.e., for all sufficiently large multiples of some $p_0$.
We say that such a $p$ \defining{computes the asymptotic multiplier ideal}.

\begin{example}\label{example: asymptotic multiplier ideal of smooth}
In the trivial case $\a_n = \a_1^n$, the asymptotic multiplier ideals reduce to the ordinary multiplier ideals:
$\multideal(\a_{\bullet}^t) = \multideal(\a_1^t)$.
This has the following consequence: If $I$ is a reduced ideal defining a smooth and irreducible variety of codimension $e$,
then
\[  \asympt{t}{I} = \multideal(I^t) = I^{\lfloor t \rfloor + 1 - e}. \]

As before, if $I$ is only reduced, then by restricting to a neighborhood of a smooth point on each
irreducible component of $\Zeros(I)$, we see that $\asympt{t}{I} \subseteq I$ for $t \geq e=\bight(I)$.
And, more generally, $\asympt{(e+k)}{I} \subseteq \symbolicpower{I}{k+1}$
for any $k \geq 0$ and any reduced ideal $I$.
\end{example}

\begin{remark}\label{rem: inclusion in asymptotic multiplier ideals}
Conversely, $\a_n \subseteq \multideal(\a_{\bullet}^n)$.
In fact, for every $n, t$, $\a_n \cdot \multideal(\a_{\bullet}^t) \subseteq \multideal(\a_{\bullet}^{t+n})$
(Theorem 11.1.19(iii) of \cite{pag2}).
This is exactly the extra piece we will add to the argument of \cite{els:symbolic-powers}
to deduce Theorem~\ref{thm:main}.
\end{remark}

As before, $\multideal(\a_{\bullet}^0) = (1)$ and if $t_1 < t_2$ then $\multideal(\a_{\bullet}^{t_1}) \supseteq \multideal(\a_{\bullet}^{t_2})$.

The asymptotic multiplier ideals satisfy subadditivity:
$\multideal(\a_{\bullet}^{t_1 + t_2}) \subseteq \multideal(\a_{\bullet}^{t_1}) \multideal(\a_{\bullet}^{t_2})$,
so $\multideal(\a_{\bullet}^{r t}) \subseteq \multideal(\a_{\bullet}^t)^r$~\cite[11.2.3]{pag2}.
This follows immediately from the subadditivity theorem for ordinary multiplier ideals.
(Let $p$ large and divisible enough compute all the asymptotic multiplier ideals appearing in the equation, then apply the ordinary
subadditivity theorem for $\a_p$.)

\subsection{Log canonical thresholds}

For an ideal $I \neq (0), (1)$, we define $\lct(I) = \sup \{ t \mid \multideal(I^t) = (1) \}$.
This is a positive rational number.
It turns out that $\multideal(I^{\lct(I)}) \neq (1)$.
(See \cite{MR2068967} or \cite{pag2}.)

Let $I$ be a radical ideal and let $e'$ be the minimum of the codimensions of the irreducible components of $\Zeros(I)$.
Then $\lct(I)$ satisfies
\[
  0 < \lct(I) \leq e' .
\]
(Restricting to a neighborhood of a general point on a codimension $e'$ component of $\Zeros(I)$,
$\multideal(I^{e'})$ vanishes on the component by Property~\ref{property: multiplier ideal of smooth}.)

For a graded system of ideals $\a_{\bullet}$, we define $\lct(\a_{\bullet}) = \sup \{ t \mid \multideal(\a_{\bullet}^t) = (1) \}$.
This may be infinite or irrational.
However for the graded system of symbolic powers of a radical ideal $I$,
we have $\lct(\symbi) \leq e'$ as above.

As shown in \cite[Remark 3.3]{MR1936888},
\[  \lct(\a_{\bullet}) = \sup p \lct(\a_p) = \lim p \lct(\a_p) . \]
Taking $p=1$, this shows $\lct(\symbi) \geq \lct(I)$ for a radical ideal $I$.

\section{Examples}\label{sect: examples}

In this section we give the asymptotic multiplier ideals of graded systems of monomial ideals,
especially for the symbolic powers of a radical (i.e., squarefree) monomial ideal,
and the asymptotic multiplier ideals of graded systems of divisor and hyperplane arrangements.

\subsection{Monomial ideals}

The following theorem gives the ordinary multiplier ideals of a monomial ideal.
\begin{thm}[\cite{howald:monomial}]
Let $I$ be a monomial ideal with Newton polyhedron $N=\Newt(I)$.
Then $\multideal(I^t)$ is the monomial ideal containing $x^v$ if and only if $v+(1,\dots,1) \in \Int(t \cdot N)$.
\end{thm}
Here $\Int( \, )$ denotes topological interior.
In particular, $\lct(I)=1/t$ where $t \cdot (1,\dots,1)$ is in the boundary of $\Newt(I)$.

Let $I_{\bullet} = \{I_p\}$ be a graded system of monomial ideals.
Let $N_p = \Newt(I_p)$.
Then $I_p^k \subseteq I_{pk}$, so $k \cdot N_p \subseteq N_{pk}$, which means $\frac{1}{p} N_p \subseteq \frac{1}{pk} N_{pk}$.
Let $N(I_{\bullet}) = \bigcup \frac{1}{p} N_p$.
Since this is an ascending union of convex sets, it is convex.
\begin{thm}[\cite{MR1936888}]\label{thm: asymptotic multiplier ideal of monomial graded system}
$\multideal( I_{\bullet}^t )$ is the monomial ideal containing $x^v$ if and only if $v+(1,\dots,1) \in \Int(t \cdot N(I_{\bullet}))$.
\end{thm}
\begin{proof}
If $p$ computes $\multideal(I_{\bullet}^t)$ and $x^v \in \multideal(I_{\bullet}^t) = \multideal(I_p^{t/p})$
then $v+(1,\dots,1) \in \Int( \frac{t}{p} N_p) \subseteq \Int(t \cdot N(I_{\bullet}))$.
Conversely if $v+(1,\dots,1) \in \Int(t \cdot N(I_{\bullet}))$ then $v+(1,\dots,1) \in \Int(\frac{t}{p} N_p)$ for some $p$,
whence $x^v \in \multideal(I_p^{t/p}) \subseteq \multideal(I_{\bullet}^t)$.
\end{proof}

For a graded system of monomial ideals, $\lct(I_{\bullet}) = 1/t$ where $t \cdot (1,\dots,1)$ is in the boundary of $N(I_\bullet)$.

More can be said in a special situation:

\begin{prop}\label{prop: Newton polyhedron of special monomial graded system}
If a graded system is given by $I_p = C_1^p \cap \dots \cap C_r^p$ for fixed monomial ideals $C_1, \dots, C_r$,
then in the above notation, $N(I_{\bullet}) = \bigcap \Newt(C_i)$.
\end{prop}
\begin{proof}
For a monomial ideal $\a$, let $\monom(\a)$ denote the set of exponent vectors of monomials in $\a$,
so that $\Newt(\a)$ is the convex hull $\conv(\monom(\a))$.
For $p \geq 1$ we have $\monom(I_p) = \bigcap \monom(C_i^p)$,
so
\[  \Newt(I_p) \subseteq \bigcap \Newt(C_i^p) = p \cdot \bigcap \Newt(C_i) .  \]
This shows $N(I_{\bullet}) \subseteq \bigcap \Newt(C_i)$.

For the reverse inclusion, note $\bigcap \Newt(C_i)$ is a rational polyhedron.
For $p$ sufficiently divisible, $p \cdot \bigcap \Newt(C_i)$ is a lattice polyhedron; in particular all its
extremal points (vertices) have integer coordinates, and $p \cdot \bigcap \Newt(C_i)$ is the convex hull
of the integer (lattice) points it contains.
So let $v$ be an integer point in $p \cdot \bigcap \Newt(C_i) = \bigcap \Newt(C_i^p)$.
Then $x^v \in C_i^p$ for each $i$, so $x^v \in \bigcap C_i^p = I_p$.
This shows $p \cdot \bigcap \Newt(C_i) \subseteq \conv(\monom(I_p))$.
Therefore $\bigcap \Newt(C_i) \subseteq \frac{1}{p} \Newt(I_p) \subseteq N(I_{\bullet})$.
\end{proof}
One can check that in the situation of the above proposition, $\lct(I_{\bullet}) = \min \lct(C_i)$.

\begin{prop}\label{prop: asymptotic multiplier ideal of symbolic powers of monomial}
Let $I=I_1$ be a reduced monomial ideal and $I_p = \symbolicpower{I}{p}$.
Suppose $I$ is not the maximal ideal.
Let $N$ be the convex region defined by the linear inequalities that correspond to unbounded facets of $\Newt(I)$.
Then $N=N(\symbolicpower{I}{\bullet})$; in particular $\multideal(t \cdot \symbolicpower{I}{\bullet})$ is the monomial ideal
containing $x^v$ if and only if $v+(1,\dots,1) \in \Int(t \cdot N)$.
\end{prop}
\begin{proof}
Let $I = C_1 \cap \dots \cap C_r$, the $C_i$ minimal primes of $I$.
Then $\symbolicpower{I}{p} = C_1^p \cap \dots \cap C_r^p$.
As long as $I$ is non-maximal, equivalently each $C_i$ is non-maximal,
the $\Newt(C_i)$, together with the facets of the positive orthant,
correspond precisely to the unbounded facets of $\Newt(I)$.
The result follows by the previous propositions.
\end{proof}
In particular, each $\lct(C_i) = \height C_i$,
so $\lct(\symbolicpower{I}{\bullet}) = \min \height C_i = e'$,
where $e'$ is the minimum codimension of any irreducible component of the variety $V(I)$.

\subsection{Hyperplane arrangements}

Let $D$ be a divisor with real (or rational or integer) coefficients.
The multiplier ideals $\multideal(t \cdot D)$ are defined similarly to the multiplier ideals of ideals.
All the properties described above hold for multiplier ideals of divisors.
In fact, when $D$ is a divisor with integer coefficients with defining ideal $I$, $\multideal(t \cdot D) = \multideal(I^t)$.
%However $\multideal(t \cdot D)$ is defined also for divisors $D$ with real coefficients.
See \cite{pag2} for details.

The multiplier ideals of hyperplane arrangements were computed in~\cite{mustata:hyperplane-arrangements},
with the following result.
\begin{thm}
Let $D = b_1 H_1 + \dots + b_r H_r$ be a weighted central arrangement, where the $H_i$ are hyperplanes in $\C^n$
containing the origin and the $b_i$ are nonnegative real numbers, the weights.
Let $L(D)$ be the intersection lattice of the arrangement $D$, the set of proper subspaces of $\C^n$
which are intersections of the $H_i$.
For $W \in L(D)$, let $r(W) = \codim(W)$ and $s(W) = \sum\{ b_i \mid W \subseteq H_i \} = \ord_W(D)$.
Then the multiplier ideals of $D$ are given by
\[
  \multideal(t \cdot D) = \bigcap_{W \in L(D)} I_W^{\lfloor t \cdot s(W) \rfloor + 1 - r(W)} ,
\]
where $I_W$ is the ideal of $W$.
\end{thm}
In fact, the intersection over $W \in L(D)$ can be reduced to an intersection over $W \in \mathcal{G}$
for certain subsets $\mathcal{G} \subseteq L(D)$ called building sets; see~\cite{teitler:hyperplane-arrangements}.
The log canonical threshold is given by $\lct(D) = \min_{W \in L(D)} r(W)/s(W)$; this may be reduced to a minimum
over $W \in \mathcal{G}$.

With this in hand it is easy to describe a similar result for graded systems of hyperplane arrangements.

We will say a \defining{graded system of divisors} is a sequence $D_{\bullet} = \{ D_p \}_{p \geq 1}$
such that $D_p + D_q \geq D_{p+q}$.
Equivalently, for each component $E$, the $\ord_E(D_p)$ satisfy
$\ord_E(D_p)+\ord_E(D_q) \geq \ord_E(D_{p+q})$.
If the $D_p$ have integer weights, then the condition of the $D_p$ forming a graded system of divisors
is equivalent to requiring the ideals $I_p = I(D_p)$ to form a graded system of ideals.
Define the asymptotic multiplier ideal $\multideal(t \cdot D_{\bullet}) = \max_p \multideal(\frac{t}{p} D_p)$,
as for graded systems of ideals.

The following lemma will be helpful:

\begin{lem}[\cite{MR1936888}, Lemma 1.4]
Let $\{a_p\}$ be a sequence of non-negative real numbers such that $a_p + a_q \geq a_{p+q}$ for all $p,q$.
Then $\frac{1}{p} a_p$ converges to a finite limit; in fact $\frac{1}{p} a_p \to \inf \frac{1}{p} a_p$.
\end{lem}

For a graded system $D_{\bullet}$ of divisors, let $D_{\infty} = \sum a_E E$
where $a_E = \lim_{p \to \infty} \frac{1}{p} \ord_E(D_p)$.

\begin{prop}
Let $D_{\bullet}$ be a graded system of divisors.
Then $\multideal(t \cdot D_{\bullet}) = \multideal(t \cdot D_{\infty})$.
\end{prop}
This follows from considering a common resolution of singularities of all the $D_p$ and $D_{\infty}$.
The following is an immediate consequence.
\begin{prop}
Let $D_{\bullet}$ be a graded system of divisors where each $D_p$ is a central hyperplane arrangement.
Let the hyperplanes be $H_1, \dots, H_r$.
Let $D_p = b_{1,p} H_1 + \dots + b_{r,p} H_r$, and let $b_{i,\infty} = \lim b_{i,p}/p$.
Let $L(D_0)$ be the intersection lattice of the (reduced) arrangement $D_0 = H_1 \cup \dots \cup H_r$,
and for $W \in L(D_0)$ let $s_{\infty}(W) = \sum \{ b_{i,\infty} : W \subseteq H_i \}$, $r(W) = \codim(W)$.
Then
\[  \multideal(t \cdot D_{\bullet}) = \bigcap_{W \in L(D_0)} I_W^{\lfloor t \cdot s_{\infty}(W) \rfloor + 1 - r(W)} = \multideal(t \cdot D_{\infty}) , \]
where $D_{\infty}$ is defined as above.
\end{prop}
Again the intersection can be reduced to $W \in \mathcal{G}$ for a building set $\mathcal{G} \subseteq L(D_0)$.
The log canonical threshold is given by $\lct(D_{\bullet}) = \lct(D_{\infty}) = \min_W r(W)/s_{\infty}(W)$.

\section{Proof of Theorem}\label{sect: proof}

At this point the theorem is easy to prove.
The real work was to develop the definition of multiplier ideals and show they have the properties
described in \S\ref{sect: mult-ideals}.

We have $\multideal(I^e) \subseteq I$. Together with the subadditivity theorem this gives the following chain of inclusions:
\[
  \multideal(I^{er}) \subseteq \multideal(I^e)^r \subseteq I^r .
\]
Unfortunately $\symbolicpower{I}{er}$ is not necessarily contained in $\multideal(I^{er})$.
We must enlarge these multiplier ideals enough to contain $\symbolicpower{I}{er}$
but not too much to destroy the containment in $I^r$.
First rewrite the above as
\[
  \multideal((I^p)^{\frac{er}{p}}) \subseteq \multideal((I^p)^{\frac{e}{p}})^r \subseteq I^r .
\]
These are the same ideals by Property~\ref{property: integer powers}.
Now let $p$ be sufficiently large and divisible and enlarge $I^p$ to $\symbolicpower{I}{p}$.
The multiplier ideals become asymptotic multiplier ideals,
and we will see in a moment that the inclusions above still hold:
\[
  \asympt{er}{I} \subseteq \asympt{e}{I}^r \subseteq I^r .
\]
By Remark~\ref{rem: inclusion in asymptotic multiplier ideals} we have
$\symbolicpower{I}{er} \subseteq \asympt{er}{I}$.
So this shows $\symbolicpower{I}{er} \subseteq I^r$.

This explains why we use asymptotic multiplier ideals rather than ordinary multiplier ideals in this proof.
We arrive at the following proof of Theorem~\ref{thm:main}.

\begin{proof}
We have the following chain of inclusions:
\begin{equation}\label{eqn:chain}\tag{$\dagger$}
\begin{split}
    \symbolicpower{I}{er+kr-\ell} &= \symbolicpower{I}{er+kr-\ell} \asympt{\ell}{I} \\
      & \subseteq \asympt{(er+kr)}{I} \\
      & \subseteq \asympt{(e+k)}{I}^r \\
      & \subseteq (\symbolicpower{I}{k+1})^r
\end{split}
\end{equation}
which is justified as follows.
For $\ell < \lct(\symbi)$, $\asympt{\ell}{I} = (1)$.
The first inclusion is Remark~\ref{rem: inclusion in asymptotic multiplier ideals}.
The second inclusion holds by the subadditivity theorem.
The last inclusion is Example~\ref{example: asymptotic multiplier ideal of smooth}.
\end{proof}

Theorem~2.2 of~\cite{els:symbolic-powers} is shown by exactly the above argument with $\ell=0$.

\section{Non-improvement}\label{sect: non-improvement}

Using ``classical'' methods, Bocci--Harbourne have given some improvements in special cases
to the Ein--Lazarsfeld--Smith theorem
that $\symbolicpower{I}{er} \subseteq I^r$ for every reduced ideal $I$ with $\bight(I)=e$.
For example~\cite{bocci-harbourne} shows the resurgence of an ideal $I$ of general points in $\pr{2}$ is at most $3/2$,
so $\symbolicpower{I}{m} \subseteq I^r$ for $m \geq 3r/2$.
However, the argument given above for the proof of Theorem~\ref{thm:main},
either via asymptotic multiplier ideals or via characteristic $p$ methods,
is the only way I am aware of to show for every reduced ideal $I$ of height $e$ that
$\symbolicpower{I}{er} \subseteq I^r$ (i.e., the resurgence is at most $e$)
or even that the resurgence is finite for every reduced ideal.

One may ask, how far can the same multiplier ideal methods be pushed to improve the bounds in
the Ein--Lazarsfeld--Smith theorem?

\subsection{Restriction of log canonical threshold}
The value $\ell$ in Theorem~\ref{thm:main} is severely restricted.
Let $e'$ be the minimum of the codimensions of the irreducible components of $\Zeros(I)$.
We saw $0 < \lct(I) \leq e'$,
but it often happens that $\lct(I)$ is much smaller than $e'$.
For $I$ a homogeneous ideal in $\C[x_1,\dots,x_n]$, we have
\[
  \frac{1}{\mult_0(I)} \leq \lct(I) \leq \frac{n}{\mult_0(I)}
\]
(\cite[9.3.2-3]{pag2}),
where $\mult_0(I)$ is the multiplicity of $I$ at the origin, equivalently, the least degree
of a nonzero form in $I$.
So if $\lct(I)>1$ then $I$ must contain a form of degree strictly less than $n$.

For ideals of reduced sets of points in $\pr{2}$ one can show the converse,
so $\lct(I)>1$ if and only if the points lie on a conic (which may be reducible).
So Theorem~\ref{thm:main} implies Harbourne's conjecture and answers Huneke's question
only for points on a conic, which (for smooth conics at least) had already been treated by
Bocci--Harbourne~\cite{bocci-harbourne-resurgence}.

We only need $\ell < \lct(\symbi)$,
which is a priori less restrictive than $\ell < \lct(I)$, but still restricts us to $\ell \leq e'-1$.
Indeed, there are radical ideals $I$ with $\lct(I) < \lct(\symbi)$.
However I do not know of an ideal $I$ such that there is an integer $\ell$, $\lct(I) \leq \ell < \lct(\symbi)$.

For a radical homogeneous ideal $I$,
\[
  \lct(\symbi) \leq \frac{n}{{\displaystyle \lim_{p \to \infty}} \ \frac{1}{p} \mult_0(\symbolicpower{I}{p}) }
\]
where the limit exists because $\mult_0(\symbolicpower{I}{p}) + \mult_0(\symbolicpower{I}{q}) \geq \mult_0 (\symbolicpower{I}{p+q})$.
If $\lct(\symbi) > 1$ then for some $p$ there must be a homogeneous form $F$ vanishing to order $p$
along the variety defined by $I$, of degree strictly less than $pn$.
This is weaker than the requirement that if $\lct(I)>1$ then $I$ must contain a form of degree less than $n$,
which is the same statement with the added condition $p=1$; but it does not seem very much weaker.

\subsection{The second inclusion}

Let $I = (xy,xz,yz) \subseteq \C[x,y,z]$ be the ideal of the union of the three coordinate axes.
Using Howald's theorem and its asymptotic version one can compute all the ideals appearing in~\eqref{eqn:chain}.
Since they are all integrally closed monomial ideals, we give them by giving their Newton polyhedra.
Here $e=2$; we take $k=0$.
First,
\[
  N_{\bullet} = \{ \, (a,b,c) \, \mid \, a+b,a+c,b+c \geq 1 \, \} \ni \Big( \frac{1}{2},\frac{1}{2},\frac{1}{2} \Big) .
\]
We have $\lct(I) = 3/2$ and $\lct(\symbi)=2$, so we take $\ell=1$.
Now,
\begin{align*}
  \Newt \left[ \symbolicpower{I}{2r-1} \right] &= \{ \, (a,b,c) \, \mid \, a+b,a+c,b+c \geq 2r-1, \, a+b+c \geq 3r-1 \, \} , \\
  \Newt \left[ \asympt{2r}{I} \right] &= \{ \, (a,b,c) \, \mid \, a+b,a+c,b+c \geq 2r-1, \, a+b+c \geq 3r-1 \, \} , \\
  \Newt \left[ \left(\asympt{2}{I}\right)^r \right] &= \{ \, (a,b,c) \, \mid \, a+b,a+c,b+c \geq r, \, a+b+c \geq 2r \, \} , \\
  \Newt \left[ I^r \right] &= \{ \, (a,b,c) \, \mid \, a+b,a+c,b+c \geq r, \, a+b+c \geq 2r \, \} .
\end{align*}
This example shows that the place where improvements are needed is the second inclusion in~\eqref{eqn:chain},
which relies on the subadditivity theorem.

\bibliography{biblio}

\end{document}